\documentclass[english,10pt]{amsart}

\usepackage{amsmath, amssymb, amsthm, enumerate}
\usepackage[all]{xy}
\usepackage[activeacute]{babel}
\usepackage[alphabetic]{amsrefs}

\newtheorem{thm}[equation]{Theorem}
\makeatletter
\let\c@subsubsection\c@equation
\makeatother

\newtheorem{prop}[equation]{Proposition}

\theoremstyle{remark}
\newtheorem{rmk}[equation]{Remark}

\theoremstyle{definition}

\newcommand{\Hom}{\mathrm{Hom}}
\newcommand{\inthomeff}{\mathbf{hom}^{\mathrm{eff}}}

\newcommand{\krulldim}{\mathrm{dim}}

\newcommand{\spec}[1]{\mathrm{Spec}(#1)}

\newcommand{\sphere}{\mathbf 1}

\newcommand{\DMeff}[1]{DM^{\mathrm{eff}}_{#1}}
\newcommand{\DM}[1]{DM_{#1}}

\newcommand{\HINST}[1]{\mathbf{HI}_{#1}}

\numberwithin{equation}{subsection}

\begin{document}


\title{Regular homomorphisms and Mixed Motives}



\author{Ivan Hernandez}
\address{Instituto de Matem\'aticas, Ciudad Universitaria, UNAM, DF 04510, M\'exico}
\email{ivan\_e@ciencias.unam.mx}

\author{Pablo Pelaez}
\address{Chebyshev Laboratory\\
St. Petersburg State University\\
14th Line V. O., 29B\\
Saint Petersburg 199178 Russia}
\email{pablo.pelaez@gmail.com}





\begin{abstract}
Let $X$ be a smooth projective variety of dimension $d$ over an
algebraically closed field $k$.
The main goal of this paper is to study, in the context of Voevodsky's
triangulated category of motives $\DM{k}$,
the group $CH^n_{\mathrm{alg}}(X)$
of codimension $n$ algebraic cycles of $X$, 
algebraically equivalent to zero, modulo
rational equivalence, $1\leq n \leq d$.

Namely, for any regular homomorphism $\psi$ (in the sense of Samuel)
defined on $CH^n_{\mathrm{alg}}(X)$,
we construct $M^n_{\psi}(X)\in \DM{k}$, which is a
reasonable approximation, with respect
to the slice filtration in $\DM{k}$,
of the motive of $X$, $M(X)$; and
a map $z_\psi : M^n_{\psi}(X)\rightarrow M(X)$ in
$\DM{k}$, which computes the kernel of $\psi$.  
 We construct as well
a map, $z_{\mathrm{ab}}^n: M^n_{\mathrm{ab}}(X) \rightarrow M(X)$
having analogue properties but which instead computes the subgroup
$CH^n_{\mathrm{ab}}(X)\subseteq
CH^n_{\mathrm{alg}}(X)$ of algebraic cycles abelian
equivalent to zero (in the sense of Samuel).
\end{abstract}

\thanks{The second author was supported by the Russian Science
Foundation grant 19-71-30002.}
 
\maketitle

\section{Introduction}  \label{sec.introd}

\subsection{}
Let $k$ be an algebraically closed base field, and $X$ a smooth projective
variety over $k$ of dimension $d$.  We will write $CH^n(X)$ for the group
of codimension $n$ algebraic cycles of $X$ modulo rational equivalence, 
$1\leq n\leq d$, and $CH^n_{\mathrm{alg}}(X)\subseteq CH^n(X)$ for the
subgroup of algebraic cycles which are algebraically equivalent to zero.

A classical method for the study of $CH^n_{\mathrm{alg}}(X)$ consists in
the analysis of
subgroups of $CH^n_{\mathrm{alg}}(X)$ which are
parametrized by abelian varieties
\cite{MR0424807}*{pp. 445-446}.  Namely, one considers the kernel of
regular homomorphisms $\psi :CH^n_{\mathrm{alg}}(X) \rightarrow A(k)$
in the sense of Samuel \cite{MR116010}*{p. 477}, \eqref{def.reg.hom};
where $A(k)$ is the group of $k$-rational
points for a given abelian variety $A$ over $k$.  Noteworthy examples are the
Albanese map $\psi = \mathrm{alb}_X : CH^d_{\mathrm{alg}}(X)
\rightarrow \mathrm{Alb}(X) (k)$ for zero cycles in $X$, the natural
homomorphism $\psi  : CH^1_{\mathrm{alg}}(X)
\rightarrow \mathrm{Pic}^0(X) (k)$ for codimension $1$ cycles in $X$, and
in case $k=\mathbb C$, $1<n<d$, Weil's Abel-Jacobi map
$\psi _n : CH^n_{\mathrm{alg}}(X)
\rightarrow J^n_a(X) (\mathbb C)$ into the algebraic part of the Weil-Griffiths 
intermediate Jacobian $J^n(X)$, see \eqref{ss.int.jac}.  The kernel of
$\psi _n$ is the subgroup $CH^n_{AJ}(X)\subseteq CH^n_{\mathrm{alg}}
(X)$ of algebraic
cycles Abel-Jacobi equivalent to zero.

In this work we study regular homomorphisms
in the context of Voevodsky's triangulated category of motives.
We will write $\DMeff{k}$ for Voevodsky's triangulated category of 
effective motives; and $\sphere$,
$M(X)\in \DMeff{k}$, respectively, 
for the motive of $\spec{k}$, $X$ \eqref{ss.DMdef}. 
Let $\mathbb P ^m$ denote the $m$-dimensional projective space over $k$,
and let $\sphere(r)[2r]\in \DMeff{k}$, $0\leq r \leq m$,  be the direct summands
in the decomposition of $M(\mathbb P ^m)$ in $\DMeff{k}$
\cite{MR1764202}*{Prop. 3.5.1}, \eqref{ss.DMdef}.

Consider a regular homomorphism
$\psi :CH^n_{\mathrm{alg}}(X) \rightarrow A(k)$.  In our first main result \eqref{thm1},
\eqref{thm2} we construct a map
$z_\psi : M^n_{\psi}(X)\rightarrow M(X)$ in  $\DMeff{k}$ which satisfies two
properties: 
\begin{enumerate}
\item It is a reasonable approximation of $M(X)$ in the following sense:
for $r\geq 1$, the maps $f_{d-n+r}(z_\psi)$, $s_{d-n+r}(z_\psi)$ are 
isomorphisms in $\DMeff{k}$, where $f_m$, $m\in \mathbb Z$, 
(resp. $s_m$) is the $(m-1)$-effective
cover (resp. the $m$-slice) of Voevodsky's slice filtration
\eqref{eq.slice.filtration}-\eqref{diag.slice.tr}.

\item It computes the kernel of $\psi$.  Namely, under
the isomorphism $\Hom_{\DMeff{k}}(\sphere (d-n)[2d-2n], M(X))\cong
CH^n(X)$ \cite{MR1764202}*{Prop. 2.1.4, Thm. 3.2.6}, the image of the induced
map $z_{\psi \ast}:\Hom_{\DMeff{k}}(\sphere (d-n)[2d-2n], M^n_{\psi}(X))
\rightarrow \Hom_{\DMeff{k}}(\sphere (d-n)[2d-2n], M(X))$ is the
kernel of the regular homomorphism $\psi : CH^n_{\mathrm{alg}}(X) \rightarrow
A(k)$.
\end{enumerate}

We also study the subgroup $CH^n_{\mathrm{ab}}(X)
\subseteq CH^n_{\mathrm{alg}}(X)$ of algebraic cycles abelian equivalent
to zero \eqref{def.ab.eq}, which consists of those elements in the kernel
of every regular homomorphism $\psi :CH^n_{\mathrm{alg}}(X) \rightarrow
A(k)$.  When $k=\mathbb C$, $CH^n_{\mathrm{alg}}(X)\subseteq
CH^n_{AJ}(X)$ and the natural map:
$CH^n_{\mathrm{alg}}(X)/CH^n_{\mathrm{ab}}(X)\rightarrow
CH^n_{\mathrm{alg}}(X)/CH^n_{AJ}(X)$ is conjecturally an isogeny
\cite{MR0424807}*{p. 447}.  The conjecture holds for $n=1$, $d$ (in this
situation both groups are equal by the classsical theory of the Picard and
Albanese varieties) and for $n=2$ \cite{MR777590}*{p. 982 Cor.},
\cite{MR805336}*{p. 229 Thm. 1.11.1} (where equality holds as well).

In our second main result \eqref{thm.1.ab}, \eqref{thm2.ab} we
construct a map $z_{\mathrm{ab}}^n: M^n_{\mathrm{ab}}(X) \rightarrow M(X)$
in $\DMeff{k}$ satisfying properties similar to those of $z_\psi$ above.
Namely,
\begin{enumerate}
\item  For $r\geq 1$,
the maps $f_{d-n+r}(z_{\mathrm{ab}}^n)$, $s_{d-n+r}(z_{\mathrm{ab}}^n)$ are 
isomorphisms in $\DMeff{k}$.
\item Under
the isomorphism $\Hom_{\DMeff{k}}(\sphere (d-n)[2d-2n], M(X))\cong
CH^n(X)$ \cite{MR1764202}*{Prop. 2.1.4, Thm. 3.2.6}, the image of the induced
map $z_{\mathrm{ab} \ast}^n:\Hom_{\DMeff{k}}(\sphere (d-n)[2d-2n], 
M^n_{\mathrm{ab}}(X))
\rightarrow \Hom_{\DMeff{k}}(\sphere (d-n)[2d-2n], M(X))$ is the subgroup
of algebraic cycles abelian equivalent to zero:
$CH^n_{\mathrm{ab}}(X)
\subseteq CH^n_{\mathrm{alg}}(X)$.
\end{enumerate}

In a future work, we will apply the results in this paper to the study
of the orthogonal filtration $F^\bullet$ on the Chow groups $CH^n(X)_{\mathbb Q}$
\cite{MR3614974}, \cite{MR4486247}, 
\cite{pelaez2025incidenceequivalenceblochbeilinsonfiltration}, which
is an unconditional finite filtration satisfying several of the axioms of the still
conjectural Bloch-Beilinson-Murre filtration.

The paper is organized as follows:  in section \ref{sec.preel} we fix the notation
and recall some results that will be necessary in the rest of the paper.
In section \ref{s.Vmotreg} we state and establish our first main result
\eqref{thm1}, \eqref{thm2} concerning regular homomorphisms
and discuss several particular examples where our
result applies: Abel-Jacobi equivalence
\eqref{ss.int.jac} and incidence equivalence \eqref{ss.inc.eq}.  
In section \ref{s.abeq} we state and prove
our second main result \eqref{thm.1.ab}, \eqref{thm2.ab} concerning
Abelian equivalence.

\section{Preliminaries}  \label{sec.preel}

In this section we set the notation for the
rest of the paper and collect together some results from the literature
which will be used to establish our results.  The results of this
section are not original.

\subsection{Definitions and Notation}	\label{subsec.defandnots}		

Let $k$ be an arbitrary base field, and $Sch_k$ be the category
of $k$-schemes of finite type.  
Given $X\in Sch_k$, we will write $X(k)$  for
the set of $k$-points of $X$, and 
$\pi _X :X\rightarrow \spec{k}$ for the structure map.
If $x\in X(k)$, we will abuse notation and also write 
$x:\spec{k}\rightarrow X$ for the corresponding map in $Sch_k$.
In case $X$ is reduced and irreducible, let $k(X)$ denote
the function field of $X$.

We will write $Sm_k$ for the full
subcategory of $Sch_k$ which consists of smooth $k$-schemes
considered as a site with the Nisnevich topology.  Let $SmProj_k$ (resp. $Abv_k$)
be the full subcategory of $Sm_k$ given by smooth projective
varieties (resp. abelian varieties).
If $X\in SmProj_k$ of dimension $d$ and
 $x\in X(k)$, we will write $[x] \in CH^d(X)$ for the
zero-cycle associated to $x$ as a closed subscheme of $X$,
and $z_0^X: X(k)\rightarrow CH^d(X)$ for the map $x\in X(k)\mapsto
[x]\in CH^d(X)$.

We will use freely the language of triangulated categories.  Our main
references will be \cite{MR1812507}, \cite{MR751966}.  
We will write $[1]$ (resp. $[-1]$) to denote the suspension (resp. 
desuspension) funtor in a triangulated category; and for $n>0$, the
composition of $[1]$ (resp. $[-1]$) iterated $n$-times will be $[n]$
(resp. $[-n]$).  If $n=0$, then $[0]$ will be the identity functor.

Given a map $f:M\rightarrow N$ in an abelian category $\mathcal C$, we
will write $\mathrm{ker}(m)$ (resp. $\mathrm{coker}(m)$) for the kernel 
(resp. cokernel) of $f$ in $\mathcal C$.

\subsection{Voevodsky's triangulated category of motives}  \label{ss.DMdef}

We will write $Cor_k$ for the Suslin-Voevodsky category of finite
correspondence over $k$ \cite{MR1744945}, \cite{MR3590347},
and $Shv^{tr}$ for the category of Nisnevich sheaves with
transfers which is an abelian category \cite{MR2242284}*{13.1}.
Given $X\in Sm_k$, let $\mathbb Z _{tr}(X)$ be the Nisnevich
sheaf with transfers represented by $X$ \cite{MR2242284}*{2.8 and 6.2}.

Let $K(Shv ^{tr})$ be the category of unbounded chain complexes
on $Shv^{tr}$ equipped with the injective model structure
\cite{MR1780498}*{Prop. 3.13}, and let $D(Shv^{tr})$ be its
homotopy category.  We will write $K^{\mathbb A ^1}(Shv ^{tr})$ for
the left Bousfield localization \cite{MR1944041}*{3.3} of $K(Shv^{tr})$
with respect to the set of maps 
$\{ \mathbb Z _{tr}(X\times _k \mathbb A ^1)[n] \rightarrow
\mathbb Z _{tr}(X)[n] : X\in Sm_k; n\in \mathbb Z \}$ induced by the
projections $p: X\times _k \mathbb A ^1 \rightarrow X$.
Voevodsky's triangulated category of effective motives $\DMeff{k}$
is the homotopy category of $K^{\mathbb A ^1}(Shv ^{tr})$
\cite{MR1764202}.

We will write $T\in K^{\mathbb A ^1}(Shv ^{tr})$ for the
chain complex $\mathbb Z _{tr}(\mathbb G _m)[1]$ 
\cite{MR2242284}*{2.12}, where $\mathbb G _m$ is the $k$-scheme
$\mathbb A ^1 \backslash \{ 0 \}$ pointed by $1$.  Let
$Spt _T (Shv ^{tr})$ be the category of symmetric $T$-spectra on
$ K^{\mathbb A ^1}(Shv ^{tr})$ equipped with the model
structure considered in \cite{MR1860878}*{8.7 and 8.11},
\cite{MR2438151}*{Def. 4.3.29}.  Voevodsky's triangulated
category of motives $\DM{k}$ is the homotopy category of
$Spt _T (Shv ^{tr})$ \cite{MR1764202}.

Given $X\in Sm_k$, let $M(X)$ be the image of 
$\mathbb Z _{tr}(X) \in D(Shv ^{tr})$ under the $\mathbb A ^1$-localization
map $D(Shv ^{tr})\rightarrow \DMeff{k}$.  We will write
$\Sigma ^\infty : \DMeff{k}\rightarrow \DM{k}$ for the suspension functor
\cite{MR1860878}*{7.3}, we will abuse notation and simply write $E$ for
$\Sigma ^\infty E$, $E\in \DMeff{k}$.  

We notice that $\DMeff{k}$ and $\DM{k}$ are tensor
triangulated categories \cite{MR2438151}*{Thm. 4.3.76 and Prop. 4.3.77}
with unit $\sphere = M(\spec{k})$.  We will write $E(1)$ for
$E\otimes M(\mathbb G _m)[-1]$, $E\in \DM{k}$ and inductively
$E(n)=(E(n-1))(1)$, $n\geq 0$.  We observe that the functor
$\DM{k}\rightarrow \DM{k}$, $E\mapsto E(1)$ is an equivalence of
categories \cite{MR1860878}*{8.10}, \cite{MR2438151}*{Thm. 4.3.38};
we will write $E\mapsto E (-1)$ for its inverse, and inductively
$E(-n)=(E(-n+1))(-1)$, $n>0$.  By convention $E(0)=E$ for
$E\in \DM{k}$.

\subsection{}
We observe that $\DM{k}$ is a compactly generated triangulated
category \cite{MR1308405}*{Def. 1.7} with the following set of
compact generators \cite{MR2438151}*{Thm. 4.5.67}:

\begin{align}  \label{eq.DMgens}
\mathcal G _{\DM{k}} = \{ M (X)(r) : X\in Sm_k; r\in \mathbb Z \} .
\end{align}

Given $m\in \mathbb Z$, we consider:

\begin{align}  \label{eq.DMeffgenstw}
\mathcal G ^{\mathrm{eff}}(m) = \{ M(X)(r): X\in Sm_k; r\geq m \}
\subseteq \mathcal G _{\DM{k}}.
\end{align}
Let $\DMeff{k} (m)$ be the smallest full triangulated
subcategory of $\DM{k}$ containing $\mathcal G ^{\mathrm{eff}}(m)$
\eqref{eq.DMeffgenstw} and closed under arbitrary (infinite)
coproducts.
We observe that $\DMeff{k} (m)$ is also a compactly
generated triangulated category with set of generators
$\mathcal G ^{\mathrm{eff}}(m)$ \cite{MR1308405}*{Thm. 2.1(2.1.1)}.

The slice filtration \cite{MR1977582}, \cite{MR2600283}*{p. 18},
\cite{MR2249535} is the following tower of triangulated subcategories
of $\DM{k}$:

\begin{align}  \label{eq.slice.filtration}
\cdots \subseteq \DMeff{k} (m+1) \subseteq
\DMeff{k} (m) \subseteq \DMeff{k} (m-1)
\subseteq \cdots
\end{align}
where the
inclusion
$i_m:\DMeff{k} (m)\rightarrow \DM{k}$ admits a right adjoint
$r_m : \DM{k}\rightarrow \DMeff{k} (m)$ which is also a triangulated
functor \cite{MR1308405}*{Thm. 4.1}.  The
$(m-1)$-effective cover of the slice filtration is defined to be
$f_m = i_m \circ r_m : \DM{k}\rightarrow \DM{k}$ 
\cite{MR1977582}, \cite{MR2600283}*{p. 18},
\cite{MR2249535}.  Furthermore, there exist canonical triangulated functors
$s_m:\DM{k}\rightarrow \DM{k}$, $m\in \mathbb Z$, equipped with
natural transformations $f_m \rightarrow s_m$ and $s_m\rightarrow [1]\circ f_{m+1}$
which fit in a natural distinguished triangle in $\DM{k}$:
\begin{align}  \label{diag.slice.tr}
f_{m+1}E\rightarrow f_m E \rightarrow s_mE, \; \; E\in \DM{k}.
\end{align}

We observe that the
suspension functor $\Sigma ^\infty : \DMeff{k} \rightarrow
\DM{k}$ induces an equivalence of categories between $\DMeff{k}$
and $\DMeff{k} (0)$ \cite{MR2804268}*{Cor. 4.10} (in case $k$
is non-perfect of characteristic $p$, it is necessary to consider 
$\mathbb Z [\frac{1}{p}]$-coefficients \cite{MR3590347}*{Cor. 4.13, Thm.
4.12 and Thm. 5.1}).

\subsection{}  \label{ss.nt.htstr}
We will consider Voevodsky's homotopy $t$-structure
$((\DMeff{k})_{\geq 0}, (\DMeff{k})_{\leq 0})$ in $\DMeff{k}$
\cite{MR1764202}*{p. 11}.  We will use the homological notation
for $t$-structures \cite{MR2438151}*{\S 2.1.3}, \cite{MR2735752}*{\S 1.3},
and write $\tau _{\geq m}$, $\tau _{\leq m}$ for the truncation
functors and $\mathbf{h}_m = [-m](\tau _{\leq m} \circ \tau _{\geq m})$.
Let $\HINST{k}$ be the abelian category of homotopy invariant
Nisnevich sheaves with transfers on $Sm_k$, which is the heart
of the homotopy $t$-structure in $\DMeff{k}$.

\subsubsection{}  \label{ss.nt.abvar}
Let $A\in SmProj_k$ be an abelian variety.  We will write
$\mathcal A \in \HINST{k}$ for the functors of points of $A$ considered
as a homotopy invariant Nisnevich sheaf with transfers
\cite{MR1957261}*{Lem. 3.2, Rmk. 3.3}.  Given a map
$f:Y\rightarrow A$ in $Sm_k$,  we will  write
$\mathcal A
_f :M(Y)\rightarrow \mathcal A$ in $\DMeff{k}$ for the image of $f$
under the isomorphism
$\Gamma (Y, \mathcal A)\cong  \Hom_{\DMeff{k}}(M(Y), \mathcal A)$
\cite{MR1764202}*{Prop. 3.1.9 and 3.2.3}.

\begin{prop}  \label{prop.kpts.sep}
With the notation and conditions of \eqref{ss.nt.abvar}.  Assume
that the base field $k$ is algebraically closed, and let
$\mathcal F \in \HINST{k}$ \eqref{ss.nt.htstr}.  Consider a map
$f:\mathcal F \rightarrow \mathcal A$ in $\HINST{k}$.  Then
$f=0$ if and only if the induced map
$f_k =0: \Gamma (\spec{k}, \mathcal F)
\rightarrow \Gamma (\spec{k}, \mathcal A)$.
\end{prop}
\begin{proof}
On the one hand, if
$f=0$, it's immediate that the induced map $f_k =0$.
On the other hand, if 
$f_k =0: \Gamma (\spec{k}, \mathcal F)
\rightarrow \Gamma (\spec{k}, \mathcal A)$, then it is
enough to show
that for any affine $X\in Sm_k$: the induced map
$f_X=0: \Gamma (X, \mathcal F)
\rightarrow \Gamma (X, \mathcal A)$.

Let $\gamma \in \Gamma (X, \mathcal F)$, so
$f_X(\gamma) \in \Gamma (X, \mathcal A)$ is given by a map
in $Sm_k$: $f_X (\gamma):X\rightarrow A$.  Since 
$k$ is algebraically closed,
in order to show that $f_X (\gamma) =0$, it suffices to see that
for every $x\in X(k)$: $f_X(\gamma) (x)=0 \in A(k)$
\cite{MR463157}*{Ch. II. Prop. 2.6}, which follows
by the naturality of $f:\mathcal F \rightarrow \mathcal A$, i.e.
the commutativity of the following diagram:
\begin{align*}
\xymatrix{ \Gamma (X, \mathcal F) \ar[r]^-{f_X} \ar[d]_-{x^\ast}
& \Gamma (X, \mathcal A) \ar[d]^-{x^\ast}\\
\Gamma (\spec{k}, \mathcal F) \ar[r]_-{f_k=0}
& \Gamma (\spec{k}, \mathcal A)}
\end{align*}
\end{proof}

\subsection{The Albanese map}  \label{ss.alb.map}
With the notation and conditions of \eqref{ss.nt.abvar}.  
Assume that the base field $k$ is algebraically closed. 
Let $X\in SmProj_k$ of dimension $d$, and
$\psi = \mathrm{alb}_X : CH^d_{\mathrm{alg}}(X)
\rightarrow \mathrm{Alb}(X) (k)$  the Albanese map.
Let $x_0\in X(k)$ and $f:X \rightarrow A=\mathrm{Alb}(X)$
the canonical map into the Albanese variety such that $f(x_0)=0$.

Consider the following maps in $\DMeff{k}$,
$\mathcal A _f: M(X)\rightarrow \mathcal A$   \eqref{ss.nt.abvar}
 and $\pi _X:M(X)\rightarrow \sphere$ 
the morphism induced by the structure map 
$\pi _X:X \rightarrow \spec{k}$.
We will write
$< \pi _X , \mathcal A _f> : M(X)\rightarrow \sphere \oplus \mathcal A$ for
the map in $\DMeff{k}$ induced by $\pi _X$, $\mathcal A _f$ on each direct
factor of $\sphere \oplus \mathcal A$, see
\cite{MR1957261}*{(7) in p.7, and Rmk. 3.3}.  Let
\begin{align*}
\xymatrix{M_T (X)\ar[r]^-{z_T} & M(X) \ar[rr]^-{< \pi _X , \mathcal A _f>} 
&& \sphere \oplus \mathcal A }
\end{align*}
be a distinguished triangle in $\DMeff{k}$.

Now, by \cite{MR1764202}*{Prop. 2.1.4, Thm. 3.2.6} we obtain the following induced
map of abelian groups:
\begin{align*} 
\begin{split}
\xymatrix@R=0.1pc{
\Hom _{\DM{k}}(\sphere , M(X)) \ar[r]^-{\pi _{X \ast}}& 
\Hom_{\DM{k}}(\sphere , \sphere)\\
\wr  ||& \wr ||\\
CH^{d}(X) & \mathbb Z\\
\gamma \ar@{|->}[r]&  
\mathrm {deg}(\gamma)}
\end{split}
\end{align*}

Therefore, we conclude that the image of
the induced morphism:
\begin{align*}
z_{T \ast}:
\Hom _{\DMeff{k}}(\sphere , M_T (X))\rightarrow \Hom _{\DMeff{k}}(\sphere ,
M(X))\cong CH^d (X)
\end{align*}
is the albanese kernel $T(X)\subseteq CH^d_{\mathrm{alg}}(X)$.  

If we consider
rational coefficients, then we observe that the decomposition of the diagonal in
\cite{MR1061525}*{7.1 Thm. 3} holds for any $X\in SmProj _k$, so, by
\cite{MR1061525}*{1.5B, 3.1 Thm.1 and 4.1 Thm. 2}
there is a direct summand  $i_T:M'(X)\rightarrow 
M(X)_{\mathbb Q}$ of $M(X)_{\mathbb Q}$
which computes the albanese kernel, i.e. the image of the induced map
$i_{T \ast}:
\Hom _{\DMeff{k}}(\sphere , M' (X))\rightarrow \Hom _{\DMeff{k}}(\sphere ,
M(X)_{\mathbb Q})\cong CH^d (X)_{\mathbb Q}$ is
$T(X)_{\mathbb Q}
\subseteq CH^d_{\mathrm{alg}}(X)_{\mathbb Q}$.
	
Our main goal is to generalize this for higher dimensional cycles and
arbitrary regular homomorphisms 
\eqref{def.reg.hom}.	
Namely, to
construct an analogue of the map $z_T: M_T (X)\rightarrow M(X)$, which
computes the kernel of a given regular homomorphism
$\psi : CH^n_{\mathrm{alg}}(X)\rightarrow A(k)$.

\section{Voevodsky's motives and regular homomorphisms}  \label{s.Vmotreg}

\subsection{}	\label{not.corr.act}
In this section we assume further that the base field $k$ is
algebraically closed.
Let  $X$, $Y\in SmProj_k$ of dimension $d$, $d_Y$, respectively,
and $\Lambda \in CH^n (Y\times X)$, $0\leq n \leq d$.
Given $\gamma \in CH^{d_Y}(Y)$,
we will write $\Lambda (\gamma)\in CH^n(X)$ for
$p_{X\ast}(p_Y^\ast(\gamma)\cdot \Lambda)$ where
$p_X:Y\times X \rightarrow X$, $p_Y:Y\times X \rightarrow Y$
are the projections.

Given $\alpha \in CH^m(Z)$, $Z\in Sm_k$, we will abuse notation
and also write $\alpha: M(Z)\rightarrow \sphere (m)[2m]$ for the
map in $\DM{k}$ corresponding to $\alpha$ under the isomorphism
\cite{MR1883180}:
\begin{align*}
\Hom _{\DM{k}}(M(Z), \sphere (m)[2m])\cong CH^m(Z).
\end{align*}

\subsubsection{}   \label{sss.not.corr}
With the notation and conditions of \eqref{not.corr.act}.  Let
$y_0\in Y(k)$, and consider $\Lambda : M(Y)\otimes M(X)
\rightarrow \sphere (n)[2n]$ in $\DM{k}$.
Dualizing $M(X)$ \cite{MR1764202}*{Thm. 4.3.7},
\cite{MR2399083}*{Prop. 6.7.1 and \S 6.7.3}  we obtain the
following map in $\DM{k}$:
\begin{align*}
\xymatrix{\Lambda _X:M(Y) \ar[r]& M(X) (-d+n)[-2d+2n]}
\end{align*}

\subsubsection{}  \label{sss.not.simp}
Given a map $f:X\rightarrow Y$
in $Sm_k$, we will write $f:M(X)\rightarrow M(Y)$ for the map
induced by $f$ in $\DM{k}$.  In particular, we will write
 $y_0:\sphere \rightarrow M(Y)$ for the map in $\DM{k}$
induced by $y_0\in Y(k)$.

Let $\bar{y}_0$ be the following composition in $\DM{k}$ \eqref{sss.not.corr}:
\begin{align*}
\bar{y}_0=\Lambda _X \circ y_0 \circ \pi _Y: 
M(Y) \rightarrow M(X)(-d+n)[-2d+2n], 
\end{align*}
where
$\pi _Y : Y\rightarrow \spec{k}$ is the structure map.

\subsubsection{}  \label{ev.corr.y0}
Consider the following map in $\DM{k}$ \eqref{sss.not.corr}-\eqref{sss.not.simp}:
\begin{align*}
\Lambda _{y_0}=\Lambda _X - \bar{y}_0 :
M(Y)\rightarrow M(X)(-d+n)[-2d+2n].
\end{align*}
Then,
combining Lieberman's lemma \cite{MR1644323}*{Prop. 16.1.1} and
\cite{MR1764202}*{Prop. 2.1.4, Thm. 3.2.6},  we deduce that
$\Lambda _{y_0}$ 
 induces the following map of abelian groups:
\begin{align} \label{action.corr}
\begin{split}
\xymatrix@R=0.1pc{
\Hom _{\DM{k}}(\sphere , M(Y)) \ar[r]^-{\Lambda _{y_0 \ast}}& 
\Hom_{\DM{k}}(\sphere , M(X)(-d+n)[-2d+2n])\\
\wr  ||& \wr ||\\
CH^{d_Y}(Y) & CH^n(X)\\
\gamma \ar@{|->}[r]&  
\Lambda (\gamma) -\mathrm {deg}(\gamma)\Lambda ([y_0])}
\end{split}
\end{align}
Thus, we obtain the composition:
\begin{align}  \label{diag.algequiv}
\begin{split}
\xymatrix@R=0.1pc{Y(k) \ar[r]^-{z_0 ^Y}& CH^{d_Y}(Y) \ar[r]^-{\Lambda _{y_0 \ast}}
& CH^n(X) \\
y \ar@{|->}[r]&[y] \ar@{|->}[r]& \Lambda ([y] -  [y_0 ])}
\end{split}
\end{align}

\subsection{Regular homomorphisms}	\label{def.reg.hom}
Let $X\in SmProj_k$ of dimension $d$,
 $1\leq n \leq d$, $A \in SmProj_k$ an
abelian variety.  A group homomorphism
$\psi : CH^n_{\mathrm{alg}}(X) \rightarrow A(k)$ is
\emph{regular} \cite{MR116010}*{p. 477}, if for every $Y\in SmProj _k$,
$\Lambda \in CH^n(Y\times X)$, $y_0 \in Y(k)$;   there exists a map
$\psi _{\Lambda, y_0}: Y\rightarrow A$ in $Sm_k$, such that
for $y\in Y(k)$: $\psi _{\Lambda, y_0}(y)
=\psi(\Lambda ([y]-[y_0]))$ \eqref{not.corr.act}.

\subsubsection{}  \label{def.gp.reg.hom}
We will write $Reg^n(X,A)$ for the abelian group of regular
homomorphisms $\psi : CH^n_{\mathrm{alg}}(X) \rightarrow A(k)$.

\subsubsection{}  \label{ind.reg.A}
With the notation and conditions of \eqref{not.corr.act} and
\eqref{sss.not.corr}.
Let  $\psi : CH^n_{\mathrm{alg}}(X) \rightarrow A(k)$ 
be a regular homomorphism
\eqref{def.reg.hom}.  Then
the map  $\psi_{\Lambda, y_0}:
Y\rightarrow A$ in $Sm_k$ \eqref{def.reg.hom} induces
the following map in $\DMeff{k}$ \eqref{ss.nt.abvar}:
\begin{align*}
\mathcal A 
_{\psi_{\Lambda, y_0}}: M(Y) \rightarrow \mathcal A.
\end{align*}
Consider the homomorphism induced by $\mathcal A 
_{\psi_{\Lambda, y_0}}$:
\begin{align*}
(\mathcal A 
_{\psi_{\Lambda, y_0}})_{\ast}:\Hom_{\DMeff{k}}(\sphere , M(Y))\cong
CH^{d_y}(Y)\rightarrow \Hom_{\DMeff{k}}(\sphere , \mathcal A )\cong 
A(k)
\end{align*}

\begin{prop}  \label{prop.indmap.ev}
Let $\beta \in \Hom_{\DMeff{k}}(\sphere , M(Y))\cong
CH^{d_y}(Y)$.  Then: 
\begin{align*}
(\mathcal A 
_{\psi_{\Lambda, y_0}})_{\ast}(\beta)=
  \psi (\Lambda (\beta) -\mathrm{deg}(\beta)
\Lambda ([y_0])) \in \Hom_{\DMeff{k}}(\sphere , \mathcal A )\cong 
A(k).
\end{align*}  
\end{prop}
\begin{proof}
Since the base field $k$ is algebraically closed,
we observe that 
$\beta =\sum _i n_i y_i$, $n_i \in \mathbb Z$, $y_i \in Y(k)$, where the
sum is considered in $\Hom _{\DMeff{k}}(\sphere , M(Y))$ 
\eqref{sss.not.simp}.  Hence,
$(\mathcal A 
_{\psi_{\Lambda, y_0}})_{\ast}(\beta) = \sum _i n_i 
(\mathcal A 
_{\psi_{\Lambda, y_0}})_{\ast}(y_i)\in A(k)$.

Now, we recall that $\mathcal A \in \HINST{k}$ is the functor of points of 
$A$ \eqref{ss.nt.abvar},  so $(\mathcal A 
_{\psi_{\Lambda, y_0}})_{\ast}(y_i) = \psi_{\Lambda, y_0}(y_i)$ for
$y_i \in Y(k)$.  Hence, since $\psi$ is a regular homomorphism
\eqref{def.reg.hom}, we conclude that:
\begin{align*}
\sum _i n_i \psi (\Lambda ([y_i])- \Lambda
([y_0])) =
\sum _i n_i \psi_{\Lambda, y_0}(y_i) =
\sum _i n_i 
(\mathcal A 
_{\psi_{\Lambda, y_0}})_{\ast}(y_i) =
(\mathcal A 
_{\psi_{\Lambda, y_0}})_{\ast}(\beta).
\end{align*}
Finally, since $\psi$ is a group homomorphism, 
we deduce that
$\sum _i n_i \psi (\Lambda ([y_i])- \Lambda
([y_0]))= \psi (\Lambda (\sum _i n_i[y_i])- \sum _i n_i \Lambda
([y_0])) = \psi (\Lambda (\beta) -\mathrm{deg}(\beta)
\Lambda ([y_0]))$, which implies the result.
\end{proof}

Combining \eqref{ev.corr.y0} and \eqref{ind.reg.A}
we obtain the following diagram in $\DM{k}$:
\begin{align}  \label{ind.reg.B}
\begin{split}
\xymatrix{M(Y) \ar[r]^-{\Lambda _{y_0}}
\ar[d]^-{\mathcal A 
_{\psi_{\Lambda, y_0}}}& M(X)(-d+n)[-2d+2n]\\
\mathcal A &}
\end{split}
\end{align}

\subsubsection{}  \label{cons.proj.gen}
Let $SmProj_{\geq 1} =\{ Y\in SmProj_k : \krulldim \; Y \geq 1 \}$
and let $\mathcal T \in \DMeff{k}$ denote the following direct sum:
\begin{align*}  
\mathcal T = 
\bigoplus \limits_{\substack{\Lambda \in CH^n(Y\times X),
  y_0\in Y(k)\\
	Y\in SmProj_{\geq 1}}} M(Y),
\end{align*}
Consider the maps in $\DM{k}$
induced by \eqref{ind.reg.B} on each direct summand of $\mathcal T$:
\begin{align}  \label{diag.starting}
\begin{split}
\xymatrix{\mathcal T \ar[r]^-{(\Lambda _{y_0})}
\ar[d]_-{(\mathcal A _{\psi _{\Lambda , y_0}})}& M(X)(-d+n)[-2d+2n]\\
\mathcal A &}
\end{split}
\end{align}

\subsubsection{}  \label{not.varphi}
To simplify the notation, we will write $\varphi _m: \DM{k}\rightarrow
\DMeff{k}$, $m \in \mathbb Z$
for the triangulated functor: $E\mapsto f_m(E)(-m)[-2m]$
\eqref{eq.slice.filtration}-\eqref{diag.slice.tr}.
By \cite{MR2600283}*{Lem. 5.9}, \cite{MR2249535}*{Prop. 1.1},
we deduce that for $E\in \DMeff{k}$, $m\geq 0$:
$\varphi _m (E)\cong \inthomeff (\sphere (m)[2m], E)$, where
$\inthomeff$ is the internal Hom-functor in $\DMeff{k}$.

Recall that  
$f_{0}(E(-m)[-2m])\cong f_m(E)(-m)[-2m]$ \cite{MR3614974}*{3.3.3(2)},
and $\mathcal T \in \DMeff{k}$ \eqref{cons.proj.gen}; so,
by the universal property \cite{MR3614974}*{3.3.1},
of the counit of the adjunction
$(i_0, r_0)$ \eqref{eq.slice.filtration},
 there is a unique map $p$ such that
the following diagram in $\DM{k}$ commutes:
\begin{align}  \label{diag.modeffcov}
\begin{split}
\xymatrix{ & \varphi _{d-n}(M(X))= f_0(M(X)(-d+n)[-2d+2n]) 
\ar[d]^-{\epsilon _0}\\
\mathcal T \ar[r]_-{(\Lambda _{y_0})} \ar[ur]^-p& M(X)(-d+n)[-2d+2n]}
\end{split}
\end{align}
and  an isomorphism:
\begin{align}
\Hom _{\DMeff{k}}(\sphere , \varphi _{d-n}(M(X)))
\overset{\cong}{\underset{\epsilon _{0 \ast}}{\rightarrow}}
\Hom _{\DM{k}}(\sphere , M(X)(-d+n)[-2d+2n]).
\end{align}
Thus, by \cite{MR1764202}*{Prop. 2.1.4, Thm. 3.2.6}:
\begin{align} \label{rmk.comp.indmap.a}
\Hom _{\DMeff{k}}(\sphere , \varphi _{d-n}(M(X))) \cong CH^{n}(X).
\end{align}

\begin{rmk}  \label{rmk.comp.indmap.b}
  Under the isomorphism 
\eqref{rmk.comp.indmap.a}, the image of 
$p_\ast : \Hom_{\DMeff{k}}(\sphere ,\mathcal T) \rightarrow \Hom_{\DMeff{k}}(\sphere , 
\varphi _{d-n}(M(X)))$ \eqref{diag.modeffcov}
is given by $CH^n_{\mathrm{alg}}(X)$.
In effect, this follows from 
\eqref{diag.modeffcov}-\eqref{rmk.comp.indmap.a},
\eqref{action.corr}-\eqref{diag.algequiv} and the definition of 
algebraic equivalence \cite{MR116010}*{\S 2.2}.
\end{rmk}

\subsubsection{}  \label{main.dist.triang}
Consider the following distinguished triangle in $\DMeff{k}$:
\begin{align*}
\xymatrix{\mathcal R \ar[r]^-r & \mathcal T \ar[r]^-p &
 \varphi _{d-n}(M(X)) .}
\end{align*}

We notice that
$\varphi _{d-n}(M(X)), \mathcal T \in (\DMeff{k})_{\geq 0}$ 
\cite{pelaez2025incidenceequivalenceblochbeilinsonfiltration}*{5.1.4}; 
so, since $\{ \mathbf{h}_i :\DMeff{k}\rightarrow
\HINST{k}, i\in \mathbb Z\}$ is a cohomological functor
 \cite{MR751966}*{Thm. 1.3.6}, we conclude that
 $\mathcal R \in (\DMeff{k})_{\geq -1}$ and that $\mathbf{h}_{-1}(\mathcal R)$
is the cokernel in $\HINST{k}$ of the map
$p_\ast : \mathbf{h}_0 (\mathcal T ) \rightarrow \mathbf{h}_0 (\varphi _{d-n}(M(X)))$.

\subsubsection{} \label{sss.h-1R.birat}
Thus, since $\mathbf{h}_0 (\varphi _{d-n}(M(X)))\in \HINST{k}$ is a birational sheaf
\cite{pelaez2025incidenceequivalenceblochbeilinsonfiltration}*{5.1.5},
we conclude that $\mathbf{h}_{-1}(\mathcal R) \in \HINST{k}$ is a
birational sheaf as well \cite{MR3737321}*{Prop. 2.6.2}.

Moreover, by combining \eqref{rmk.comp.indmap.a} with \eqref{rmk.comp.indmap.b} 
we conclude that
$\Gamma (\spec{k}, \mathbf{h}_{-1}(\mathcal R))$ 
is the Neron-Severi group of
codimension $n$ algebraic cycles of $X$,
$NS^n(X)=CH^n(X)/CH^n_{\mathrm{alg}}(X)$, 
and that there exists a map in $\HINST{k}$:
$\mathbf{h}_{-1}(\mathcal R )\rightarrow \mathrm{NS}^n_{/X}$  which becomes an isomorphism after taking
global sections $\Gamma (\spec{k}, -)$,
where $\mathrm{NS}^n_{/X}$
is the Neron-Severi sheaf of Ayoub and Barbieri-Viale
\cite{MR2494373}*{Def. 3.1.2}.

\subsubsection{}  \label{diag.const.ext.1}
Now, we will apply
the homotopy $t$-structure
 \eqref{ss.nt.htstr} to construct a factorization of 
the map $(\mathcal A _{\psi _{\Lambda, y_0}})\circ r:
\mathcal R \rightarrow \mathcal A$.
So, we
consider the following diagram in $\DMeff{k}$ \eqref{diag.starting},
\eqref{main.dist.triang}:

\begin{align}  \label{diag.principal}
\begin{split}
\xymatrix{\tau _{\geq 1}\mathcal R \ar[r]^-{t_1}&\tau _{\geq 0}\mathcal R 
\ar[r]^-{t_0} \ar[d]_-{\sigma _0}
& \tau _{\geq -1}\mathcal R \cong \mathcal R 
\ar[d]_-{\sigma _{-1}}  \ar[r]^-r& \mathcal T
\ar[dd]^-{(\mathcal A _{\psi _{\Lambda , y_0}})} \\
& \mathbf{h}_{0} (\mathcal R ) \ar@{-->}[drr]_-{a_0}& \mathbf{h}_{-1} (
\mathcal R )[-1] 
\ar@{-->}[dr]^-{e_\psi}&  \\
&&& \mathcal A}
\end{split}
\end{align}
where $\tau _{i+1}\mathcal R \stackrel{t_{i+1}}{\longrightarrow}
\tau _i \mathcal R \stackrel{\sigma _i}{\longrightarrow}
\mathbf h _i (\mathcal R) [i]$ are distinguished triangles in $\DMeff{k}$
for $i=0, -1$.

We recall that $\mathcal A \in \HINST{k}$ \eqref{ss.nt.abvar}, so:
\begin{align*}
\Hom _{\DMeff{k}}(\tau _{\geq 1}\mathcal R, \mathcal A)=0=
\Hom _{\DMeff{k}}((\tau _{\geq 1}\mathcal R )[1], \mathcal A).
\end{align*}
\subsubsection{}  \label{1st.factor}
Thus, there is a unique map $a_0:\mathbf{h}_0 \rightarrow \mathcal A$
in $\DMeff{k}$, such that $a_0\circ \sigma _0 = (\mathcal A _{\psi _
{\Lambda , y_0}})\circ r \circ t_0$ in \eqref{diag.principal}.

\begin{prop}  \label{prop.main.tech}
With the notation and conditions of \eqref{ind.reg.A}, \eqref{diag.principal}
and \eqref{1st.factor}.
The map 
$a_0: \mathbf{h}_0 ( \mathcal R ) \rightarrow \mathcal A$ is zero.
\end{prop}
\begin{proof}
By \eqref{prop.kpts.sep}, it suffices to see that the induced map
$a_{0 \ast}:
\Hom _{\DMeff{k}}(\sphere , \mathbf{h}_0 (\mathcal R ) ) \rightarrow
\Hom _{\DMeff{k}}(\sphere , \mathcal A)\cong A(k)$ is zero.

We observe that 
$\sigma _{0\ast}:
\Hom _{\DMeff{k}}(\sphere , \tau _{\geq 0} \mathcal R ) 
\stackrel{\cong}{\rightarrow}
\Hom _{\DMeff{k}}(\sphere , \mathbf{h}_0 (\mathcal R ) )$, since $\spec{k}$
is a point for the Nisnevich topology; so, it is
enough to show that the induced map:
\begin{align*}
((\mathcal A _{\psi _
{\Lambda , y_0}})\circ r)_\ast:
\Hom _{\DMeff{k}}(\sphere , \mathcal R ) \rightarrow
\Hom _{\DMeff{k}}(\sphere , \mathcal A)
\end{align*}
is zero, since
$a_0\circ \sigma _0 = (\mathcal A _{\psi _
{\Lambda , y_0}})\circ r \circ t_0$ \eqref{1st.factor}.

Now, let
$\gamma :\sphere \rightarrow \mathcal R$.
On the one hand,
by \eqref{main.dist.triang}:
$0=p\circ r \circ \gamma :\sphere \rightarrow \varphi _{d-n}(M(X))$,
so, by the commutativity of \eqref{diag.modeffcov} we conclude that
 $0=(\Lambda _{y_0})
\circ r \circ \gamma :\sphere \rightarrow M(X)(-d+n)[-2d+2n]$.
On the other hand,
by the compactness of $\sphere \in \DMeff{k}$, we deduce that the
map $r\circ \gamma$
\eqref{diag.principal}, \eqref{cons.proj.gen}, factors through 
a finite direct sum:
$\oplus _{i=1}^{m} M(Y_i)$, where $(Y_i \in SmProj _k,
\Lambda _i \in CH^n(Y_i \times X), y_{0i}\in Y_i(k))$.  

Thus, we may
assume that
$r\circ \gamma =<\beta _i >: \sphere \rightarrow \oplus _{i=1}^m M(Y_i) $,
where $d_i=\dim Y_i$ and
$\beta _i \in \Hom _{\DMeff{k}}(\sphere, M(Y_i))\cong CH^{d_i}(Y_i)$.
Hence, by \eqref{action.corr}:
$\sum _{i=1}^{m} \Lambda _i (\beta _i) -\mathrm{deg} (\beta _i)
\Lambda _i ([y_{0i}]) =0\in CH^n_{\mathrm{alg}}(X)$,
since $0=(\Lambda _{y_0})
\circ r \circ \gamma = \sum _{i=1} ^m \Lambda _{y_{0i}}\circ r \circ \gamma$.

This implies that: 
\begin{align*}
\sum _{i=1}^{m} \psi (\Lambda _i (\beta _i) -\mathrm{deg} (\beta _i)
\Lambda _i ([y_{0i}])) &=\psi(\sum _{i=1}^{m} \Lambda _i (\beta _i) -\mathrm{deg} (\beta _i)
\Lambda _i ([y_{0i}]))\\
& =\psi (0) =0 \in A(k).
\end{align*}
Then, by \eqref{prop.indmap.ev}, we conclude that
$0=\sum _{i=1}^{m} \mathcal A
_{\psi_{\Lambda _i, y_{0i}}} \circ \beta _i  :\sphere \rightarrow 
\mathcal A$, which implies the result, since:
$((\mathcal A_{\psi _{\Lambda, y_0}})\circ r)_{\ast}(\gamma)=
(\mathcal A_{\psi _{\Lambda, y_0}})\circ r \circ \gamma  =
\sum _{i=1}^{m} \mathcal A
_{\psi_{\Lambda _i, y_{0i}}} \circ \beta _i :
\sphere \rightarrow \mathcal A$.
\end{proof}

\subsubsection{}  \label{ss.constr.fac}
With the notation and conditions of \eqref{ind.reg.A}, \eqref{cons.proj.gen},
\eqref{diag.starting}
and \eqref{diag.principal}.
We observe that $\Hom _{\DMeff{k}}((\tau _{\geq 0}\mathcal R )[1], \mathcal A)=0$,
since $\mathcal A \in \HINST{k}$.  Thus, combining \eqref{1st.factor} and
\eqref{prop.main.tech}, we conclude that
there is a unique  map
$e_{\psi}:  \mathbf{h}_{-1} (\mathcal R )[-1]  \rightarrow \mathcal A$,
such that the following diagram  in $\DM{k}$
commutes \eqref{diag.principal}:
\begin{align}  \label{diag.comm.square}
\begin{split}
\xymatrix{\mathcal R  \ar[r]^-r \ar[d]_-{\sigma _{-1}}& \mathcal T  
\ar[d]^-{(\mathcal A _{\psi _{\Lambda , y_0}})} \\
\mathbf{h}_{-1} (\mathcal R )[-1]  \ar[r]_-{e_\psi}& \mathcal A }
\end{split}
\end{align}
Then, by the octahedral axiom \cite{MR751966}*{Prop. 1.1.11},  we obtain
the following commutative diagram where the rows and columns
are distinguished triangles in $\DMeff{k}$ \eqref{main.dist.triang}, \eqref{diag.principal}:
\begin{align}  \label{diag.eff.regmot}
\begin{split}
\xymatrix@C=1.4pc{
\tau _{\geq 0} \mathcal R \ar[r] \ar[d]_-{t_0}& \mathcal T _{\psi} \ar[r] \ar[d]& 
 \mathcal K _{ \psi}  \ar[d]^-{k_\psi} \\
\mathcal R  \ar[r]^-r \ar[d]_-{\sigma _{-1}}& \mathcal T  \ar[r]^-p 
\ar[d]^-{(\mathcal A _{\psi _{\Lambda , y_0}})}
 & f_0(M(X)(-d+n)[-2d+2n]) \cong \varphi _{d-n}(M(X))\ar[d]^-{\Psi} \\
\mathbf{h}_{-1} (\mathcal R )[-1]  \ar[r]_-{e_\psi}& \mathcal A \ar[r]_-{a_\psi}&
\mathcal E _{\psi}}
\end{split}
\end{align}

\subsubsection{}  \label{ss.ext.HINST}
We observe that $\mathcal E _{\psi} \in \HINST{k}$,
since $\{ \mathbf{h}_i :\DMeff{k}\rightarrow
\HINST{k}, i\in \mathbb Z\}$ is a cohomological functor
 \cite{MR751966}*{Thm. 1.3.6} and $\mathcal A , \mathbf h _{-1}(\mathcal R ) \in
 \HINST{k}$.  Furthermore,  we notice that 
 $\mathcal A$ (being the functor of points of an abelian variety)
and $\mathbf h _{-1}(\mathcal R)$ 
 are birational sheaves \eqref{sss.h-1R.birat}, so,
by \cite{MR3737321}*{Prop. 2.6.2}, we deduce that
 $\mathcal E _{\psi}$ is also a birational sheaf.
 
\begin{rmk}
The commutative diagram \eqref{diag.eff.regmot} provides a
motivic presentation for a
regular homomorphism $\psi: CH^n_{\mathrm{alg}}(X)\rightarrow A(k)$
\eqref{def.reg.hom}.
\end{rmk}

Now, we consider the following commutative
diagram of abelian groups induced by \eqref{diag.eff.regmot}:
\begin{align}  \label{diag.ev.regmot}
\begin{split}
\xymatrix@C=1.7pc{  \Hom_{\DMeff{k}}(\sphere , \mathcal T) \ar[r]^-{p_\ast} 
\ar[d]^-{(\mathcal A _{\psi _{\Lambda , y_0}})_\ast}
& \Hom_{\DMeff{k}}(\sphere , \varphi _{d-n}(M(X))) \ar[d]^-{\Psi _\ast} 
\overset{(1)}{\cong} CH^n(X)\\ 
A(k)\underset{(2)}{\cong} \Hom_{\DMeff{k}}(\sphere ,\mathcal A) \ar[r]_-{a_{\psi \ast}} & 
\Hom_{\DMeff{k}}(\sphere , \mathcal E _{\psi}) \cong \Gamma (k , 
\mathcal E _{\psi})}
\end{split}
\end{align}
where the isomorphisms $(1)$, $(2)$ follow, respectively, from
\eqref{rmk.comp.indmap.a} and
the fact that $\mathcal A \in \HINST{k}$ is the functor of points
of the abelian variety $A$ \eqref{ss.nt.abvar}.  Thus,
combining \eqref{rmk.comp.indmap.b} with
\eqref{action.corr} and \eqref{prop.indmap.ev},
we obtain the following commutative diagram of abelian groups:

\begin{align}  \label{diag.ind.reghom}
\begin{split}
\xymatrix{0 \ar[r]& CH^n_{\mathrm{alg}}(X) \ar[r] \ar[d]_-\psi& CH^n (X) 
\ar[d]^-{\Psi _\ast} \\
0 \ar[r]& A(k) \ar[r]_-{a_{\psi \ast}} & \Gamma (k, \mathcal E _{\psi})}
\end{split}
\end{align}
where $\psi$ is the original regular homomorphism \eqref{ind.reg.A}, and the
injectivity of $a_{\psi \ast}$ follows from the distinguished triangle in the bottom
row of
 \eqref{diag.eff.regmot}, since
$0= \Hom _{\DMeff{k}}(\sphere , \mathbf h _{-1}(\mathcal R )[-1])$.

\begin{prop}  \label{prop.comp.ker}
With the notation and conditions of \eqref{ind.reg.A}-\eqref{diag.ind.reghom}.
Then, $\mathrm{ker}(\psi) = \mathrm{ker}(\Psi _\ast)$.
\end{prop}
\begin{proof}
By the commutativity of \eqref{diag.ind.reghom}, we conclude that 
$\mathrm{ker}(\psi) \subseteq \mathrm{ker}(\Psi _\ast)$.

On the other hand, let $\gamma \in \mathrm{ker}(\Psi _\ast)\subseteq
\Hom_{\DMeff{k}}(\sphere , \varphi _{d-n}(M(X)))$.  By construction, 
\eqref{diag.ind.reghom} is induced by \eqref{diag.ev.regmot}, 
so, by the injectivity of $a_{\psi \ast}$  \eqref{diag.ind.reghom}
it suffices to see that $\gamma$ is in
the image of $p_\ast$ in \eqref{diag.ev.regmot}. 

Now, we recall that the right column in
\eqref{diag.eff.regmot} is a distinguished triangle in $\DMeff{k}$, 
so, we deduce that
$\gamma$ is in the image of $k_{\psi \ast}: \Hom _{\DMeff{k}}(\sphere ,
\mathcal K _\psi ) \rightarrow \Hom _{\DMeff{k}}(\sphere , \varphi _{d-n}(M(X)))$.
Then, we observe 
that the top row of the commutative
diagram \eqref{diag.eff.regmot} is a distinguished triangle in $\DMeff{k}$, so, 
in order to show that $\gamma$ is in
the image of $p_\ast$ in \eqref{diag.ev.regmot}, it suffices to prove that:
$\Hom _{\DMeff{k}}(\sphere , \tau _{\geq 0} \mathcal R [1])=0$, which holds,
since
$\spec{k}$ is a point for the Nisnevich topology.
\end{proof}

\subsubsection{}  \label{ss.not.thm1}
With the notation and conditions of \eqref{ind.reg.A} and
\eqref{diag.eff.regmot}.  We will write $M^n_{\psi}(X) \in \DMeff{k}$ for 
$\mathcal K _{\psi} (d-n)[2d-2n]$ \eqref{diag.eff.regmot}, and let
$c_{\psi}:M^n_{\psi}(X)\rightarrow  f_{d-n}M(X)$ be the  map $k_\psi (d-n)[2d-2n]$
in $\DMeff{k}$ (see \eqref{diag.eff.regmot} and
\cite{MR3614974}*{3.3.3(2)}).

We define $z_\psi : M^n_{\psi}(X)\rightarrow M(X)$ to be the following composition:
\begin{align*}
\xymatrix{M^n_{\psi}(X) \ar[r]^-{c_\psi}& f_{d-n}M(X) \ar[r]^-{\epsilon _{d-n}^{M(X)}}&
 M(X)}
\end{align*}
where $\epsilon _{d-n}$ is the counit of the adjunction 
$(i_{d-n}, r_{d-n})$ \eqref{eq.slice.filtration}.  Now, we consider
the  morphism of abelian groups induced by $z _\psi$:
\begin{align}  \label{diag.main.indmap}
\begin{split}
\xymatrix{\Hom _{\DMeff{k}}(\sphere (d-n)[2d-2n],
M^n_{\psi}(X)) \ar[d]^-{z _{\psi \ast}} \\
\Hom _{\DMeff{k}}(\sphere (d-n)[2d-2n], M(X))}
\end{split}
\end{align}

\begin{thm}  \label{thm1}
With the notation and conditions of \eqref{ind.reg.A}-\eqref{diag.main.indmap}.
Under the isomorphism
$\Hom _{\DMeff{k}}(\sphere (d-n)[2d-2n], M(X))\cong CH^n(X)$
\cite{MR1764202}*{Prop. 2.1.4, Thm. 3.2.6},
the image of
$z _{\psi \ast}$  \eqref{diag.main.indmap} is $\mathrm{ker}(\psi)\subseteq CH^n_{\mathrm{alg}}(X)$.
\end{thm}
\begin{proof}
We observe that  the right column in
\eqref{diag.eff.regmot} is a distinguished triangle in $\DMeff{k}$.  Thus,
since $z_\psi = \epsilon _{d-n}^{M(X)} \circ k_\psi (d-n)[2d-2n]$,
we deduce the result by combining
 \eqref{prop.comp.ker}, Voevodsky's cancellation theorem
\cite{MR2804268} and the universal property of the counit 
$\epsilon _{d-n}^{M(X)}:f_{d-n}M(X)\rightarrow M(X)$ \cite{MR3614974}*{3.3.1}. 
\end{proof}

\begin{rmk}
Given $X \in SmProj_k$ and a regular homomorphism $\psi: CH^n_{\mathrm{alg}}
(X)\rightarrow A(k)$ \eqref{def.reg.hom},  
the commutative diagram \eqref{diag.comm.square} in $\DMeff{k}$ is canonical.
In contrast, $\mathcal K _\psi \in \DMeff{k}$ \eqref{diag.eff.regmot} and
$M^n _\psi (X)\in \DMeff{k}$ \eqref{ss.not.thm1}
are only unique up to a non-canonical isomorphism in $\DMeff{k}$.
\end{rmk}

\subsection{Intermediate Jacobians}  \label{ss.int.jac}
With the notation and conditions of \eqref{thm1}.
Assume further that $k=\mathbb C$.  Let $J^n(X)$ be the Weil-Griffiths
$n$-th intermediate Jacobian \cite{MR0050330}*{\S IV.24-26},
\cite{MR0233825}*{Ex. 2.1} (for a comparison
see \cite{MR0233825}*{Thm. 2.54}).  Consider
Weil's
Abel-Jacobi map $\psi _n : CH^n_{\mathrm{alg}}(X)
\rightarrow J^n (X)$ \cite{MR0050330}*{\S IV.27}.  The kernel of $\psi _n$ is the
subgroup $CH^n_{AJ}(X)\subseteq CH^n_{\mathrm{alg}}(X)$ of algebraic cycles
Abel-Jacobi equivalent to zero.

Furthermore, Lieberman shows
\cite{MR238857}*{Thm. 6.5(iii)} that the image of $\psi _n$ is given by the $\mathbb C$-points of an
abelian variety $J^n_a(X)$
\cite{MR238857}*{Defs. 5.6, 4.16}, and that
$\psi _n : CH^n_{\mathrm{alg}}(X)
\rightarrow J^n_a (X)(\mathbb C)$ is a regular homomorphism
\cite{MR238857}*{Thm. 6.5(i)}, \eqref{def.reg.hom}.  

\begin{thm}
With the notation and conditions of \eqref{ss.int.jac}.
There exists
 $M^n_{AJ}(X)\in \DMeff{k}$, and a map
$z_{AJ}^n: M^n_{AJ}(X) \rightarrow M(X)$ in $\DMeff{k}$,
such that under the isomorphism
$\Hom _{\DMeff{k}}(\sphere (d-n)[2d-2n], M(X))\cong CH^n(X)$
\cite{MR1764202}*{Prop. 2.1.4, Thm. 3.2.6},
the image of the induced map:
\begin{align*}
\xymatrix{ \Hom _{\DMeff{k}}(\sphere (d-n)[2d-2n], M^n_{AJ}(X)) 
\ar[d]^-{z _{AJ \ast}^n} \\
\Hom _{\DMeff{k}}(\sphere (d-n)[2d-2n], M(X))\cong CH^n(X)}
\end{align*}
is the subgroup of algebraic cycles Abel-Jacobi equivalent to
zero: $CH^n_{AJ}(X)
\subseteq CH^n_{\mathrm{alg}}(X)$.
\end{thm}
\begin{proof}
We consider the regular homomorphism
$\psi _n : CH^n_{\mathrm{alg}}(X)
\rightarrow J^n_a (X)(\mathbb C)$.  Now let
$M^n_{AJ}(X)= M^n_{\psi _n}(X) \in \DMeff{k}$ \eqref{ss.not.thm1} and
$z_{AJ}^n=z_{\psi _n}: M^n_{AJ}(X) \rightarrow M(X)$ \eqref{ss.not.thm1}.
Then, the result follows from \eqref{thm1}.
\end{proof}

\subsection{Incidence equivalence}  \label{ss.inc.eq}
With the notation and conditions of \eqref{thm1}.  Let
 $\alpha \in CH^n_{\mathrm{alg}}(X)$.  We say that
$\alpha$ is \emph{incident equivalent} to zero \cite{MR0309937}*{p. 6-7},
if for every
$Y\in SmProj_k$ and every $\beta \in CH^{d-n+1}(X\times Y)$:
$p_{Y\ast}(p_X ^\ast (\alpha)\cdot \beta)=0 \in CH^1(Y)$, where
$p_X :X\times Y \rightarrow X$, $p_Y : X\times Y \rightarrow Y$ are the
projections.  We will write $CH^n_{\mathrm{inc}}(X)\subseteq CH^n_{\mathrm{alg}}(X)$
for the subgroup of algebraic cycles incident equivalent to zero.

Furthermore, it is known
(see Grothendieck-Kleiman \cite{MR0424807}*{p. 446-447}
and H. Saito \cite{MR542191}*{Thm. 3.6}) that
there exists an abelian variety $Pic ^n(X)$ and a surjective
regular homomorphism
$\psi _n : CH^n_{\mathrm{alg}}(X)\rightarrow Pic ^n (X)(k)$, such that
$\mathrm{ker} (\psi _n)= CH^n_{\mathrm{inc}}(X)$.

\begin{thm}
With the notation and conditions of \eqref{ss.inc.eq}.
There exists
 $M^n_{\mathrm{inc}}(X)\in \DMeff{k}$, and a map
$z_{\mathrm{inc}}^n: M^n_{\mathrm{inc}}(X) \rightarrow M(X)$ in $\DMeff{k}$,
such that under the isomorphism
$\Hom _{\DMeff{k}}(\sphere (d-n)[2d-2n], M(X))\cong CH^n(X)$
\cite{MR1764202}*{Prop. 2.1.4, Thm. 3.2.6},
the image of the induced map:
\begin{align*}
\xymatrix{ \Hom _{\DMeff{k}}(\sphere (d-n)[2d-2n], M^n_{\mathrm{inc}}(X)) 
\ar[d]^-{z _{\mathrm{inc} \ast}^n} \\
\Hom _{\DMeff{k}}(\sphere (d-n)[2d-2n], M(X))\cong CH^n(X)}
\end{align*}
is the subgroup of algebraic cycles incident equivalent to zero:
$CH^n_{\mathrm{inc}}(X)
\subseteq CH^n_{\mathrm{alg}}(X)$.
\end{thm}
\begin{proof}
Consider the regular homomorphism
$\psi _n : CH^n_{\mathrm{alg}}(X)
\rightarrow Pic^n(X)(k)$ \cite{MR0424807}*{p. 446-447},
\cite{MR542191}*{Thm. 3.6}, and define $M^n_{\mathrm{inc}}(X)$,
$z_{\mathrm{inc}}^n$ to be, respectively,
$M^n_{\psi _n}(X) \in \DMeff{k}$ \eqref{ss.not.thm1},
$z_{\psi _n}: M^n_{\mathrm{inc}}(X) 
\rightarrow M(X)$ \eqref{ss.not.thm1}.  Then, the
result follows from  \eqref{thm1}.
\end{proof}

\subsection{Effective covers and slices}  \label{sec.not.thm2}
Applying the triangulated functor $f_m$ \eqref{eq.slice.filtration}
to the distinguished triangle in 
the right column of \eqref{diag.eff.regmot}, 
we obtain the following distinguished triangle in $\DM{k}$:
\begin{align}  \label{diag.dist.tr.thm2}
\xymatrix{f_m(\mathcal K _\psi)\ar[r]^-{f_m(k_\psi)} &f _m(\varphi _{d-n}(M(X)))
\ar[r]^-{f_m(\Psi)} & f_m (\mathcal E _\psi).}
\end{align}

\begin{thm}  \label{thm2}
With the notation and conditions of \eqref{ind.reg.A}-\eqref{diag.main.indmap} and
\eqref{diag.dist.tr.thm2}.
\begin{enumerate}
\item \label{thm2.a}For $m\geq 1$, $f_m(\mathcal E _\psi)\cong 0$.
\item  \label{thm2.b}For $r\geq 1$,  the map $f_{d-n+r}(z _\psi ): 
f_{d-n+r}(M^n_{\psi}(X))  \rightarrow
f_{d-n+r}(M(X))$ is an isomorphism in $\DMeff{k}$
\eqref{ss.not.thm1}.
\item  \label{thm2.c}For $r\geq 1$, the map $s_{d-n+r}(z_\psi ):
s_{d-n+r}(M^n_{\psi}(X))  \rightarrow
s_{d-n+r}(M(X))$ is an isomorphism in $\DMeff{k}$
\eqref{diag.slice.tr}.
\end{enumerate}
\end{thm}
\begin{proof}
\eqref{thm2.a}:  Since $\DMeff{k} (m) \subseteq \DMeff{k} (1)$ for $m\geq 1$ \eqref{eq.slice.filtration},
we conclude that
$f_m \cong f_m \circ f_1$, so, it is enough to prove
that $f_1(\mathcal E _\psi)\cong 0$.  
Then, we recall that
$\mathcal E _\psi \in \HINST{k}$ is a birational sheaf
\eqref{ss.ext.HINST},
so, by \cite{MR1764200}*{Thm. 5.7} we conclude that for every $Z\in Sm_k$,
$i>0$: $H^i_{\mathrm{Nis}}(Z, \mathcal E _\psi)=0$.  Thus,
by combining 
\cite{MR3737321}*{Cor. 4.5.3 and Lem. 4.5.4} and
\cite{MR2600283}*{Lem. 5.9}, \cite{MR2249535}*{Prop. 1.1}, we deduce that
$f_1(\mathcal E _\psi)=0$, as we wanted.

\eqref{thm2.b}:  
By \eqref{diag.dist.tr.thm2} and \ref{thm2}\eqref{thm2.a}
we conclude that
$f_r(k_\psi)$ is an isomorphism.
Then, we recall that by construction $z_\psi =\epsilon ^{M(X)}_{d-n} \circ
k_\psi (d-n)[2d-2n]$,
$M^n _\psi (X)=\mathcal K _\psi (X)(d-n)[2d-2n]$ \eqref{ss.not.thm1}
and $\varphi _{d-n}(M(X))=f_{d-n}(M(X))(-d+n)[-2d+2n]$
\eqref{not.varphi}, so, by \cite{MR3614974}*{3.3.3(2)} we deduce that
$f_{d-n+r}(k_\psi (d-n)[2d-2n])$ is an isomorphism in $\DMeff{k}$.  Finally,
by the universal property of the counit $\epsilon _{d-n}$  \cite{MR3614974}*{3.3.1}
we conclude that $f_{d-n+r}(\epsilon ^{M(X)}_{d-n})$ is an isomorphism
in $\DMeff{k}$, which implies the result.

\eqref{thm2.c}:  This follows from \ref{thm2}\eqref{thm2.b},
since $s_m\cong s_m\circ f_m$ for $m\in \mathbb Z$.
\end{proof}

\section{Voevodsky's motives and cycles abelian equivalent to zero}  \label{s.abeq}

\subsection{Abelian equivalence}  \label{def.ab.eq}
Let $X\in SmProj_k$ of dimension $d$, 
$1\leq n\leq d$ and $\alpha \in CH^n_{\mathrm{alg}}(X)$.
We say that $\alpha$ is \emph{abelian equivalent to zero}
\cite{MR116010}*{p. 477} if $\alpha \in \ker (\psi)$,
for every abelian variety $A\in SmProj_k$ and every
regular homomorphism $\psi : CH^n_{\mathrm{alg}}(X) \rightarrow A(k)$ 
\eqref{def.reg.hom}.

We will write $CH^n_{\mathrm{ab}}(X)$ for the subgroup of
$CH^n_{\mathrm{alg}}(X)$ consisting of algebraic cycles
abelian equivalent to zero.

\subsubsection{}  \label{ss.1st.step.abeq}
Let $\psi: CH^n_{\mathrm{alg}}(X) \rightarrow A(k)$ be a regular
homomorphism \eqref{def.reg.hom}, and with the notation and conditions of \eqref{ind.reg.A}, \eqref{cons.proj.gen},
\eqref{diag.starting}
and \eqref{diag.principal}. Consider the commutative diagram
\eqref{diag.comm.square} in $\DMeff{k}$:
\begin{align*}
\xymatrix{\mathcal R  \ar[r]^-r \ar[d]_-{\sigma _{-1}}& \mathcal T 
\ar[d]^-{(\mathcal A _{\psi _{\Lambda , y_0}})} \\
\mathbf{h}_{-1} (\mathcal R )[-1]  \ar[r]_-{e_\psi }&  \mathcal A }
\end{align*}

Let $\mathfrak{Ab}\in \HINST{k}\subseteq \DMeff{k}$ be the product
\eqref{def.gp.reg.hom}:
\begin{align}  \label{def.prod.abvar}
\mathfrak{Ab}=
\prod  \limits_{\substack{\psi \in Reg ^n (X,A)\\ A \in Abv_k}}\mathcal A
\end{align}
where $Abv_k$ is the full subcategory of $SmProj_k$ consisting of
abelian varieties and
$\mathcal A$ is the functor of points for an abelian
variety $A$ \eqref{ss.nt.abvar}.

Now, we consider the following maps in $\DMeff{k}$
induced by \eqref{diag.comm.square} on
each direct factor of $\mathfrak{Ab}$:
\begin{align}
\begin{split}
<(\mathcal A _{\psi _{\Lambda , y_0}})>:\mathcal T
\rightarrow \mathfrak{Ab}\\
<e_\psi >: \mathbf{h}_{-1} (\mathcal R )[-1] \rightarrow \mathfrak{Ab}
\end{split}
\end{align}
Then, since \eqref{diag.comm.square} commutes, we obtain
the following commutative diagram in $\DMeff{k}$:
\begin{align}
\begin{split}
\xymatrix{\mathcal R  \ar[r]^-r \ar[d]_-{\sigma _{-1}}& \mathcal T 
\ar[d]^-{<(\mathcal A _{\psi _{\Lambda , y_0}})>} \\
\mathbf{h}_{-1} (\mathcal R )[-1]  \ar[r]_-{<e_\psi >}& 
\mathfrak{Ab} }
\end{split}
\end{align}
Thus, by the octahedral axiom \cite{MR751966}*{Prop. 1.1.11},  we obtain
the following commutative diagram where the rows and columns
are distinguished triangles in $\DMeff{k}$ \eqref{main.dist.triang}, \eqref{diag.principal}:
\begin{align}  \label{diag.main.abeq}
\begin{split}
\xymatrix@C=1.4pc{ \tau _{\geq 0} \mathcal R \ar[d]_-{t_0} \ar[r] & 
\mathcal T _{\mathrm{ab}}
\ar[r] \ar[d]& \mathcal K _{\mathrm{ab}}^n(X) \ar[d]^-{k^n_{\mathrm{ab}}} \\
\mathcal R  \ar[r]^-r \ar[d]_-{\sigma _{-1}}& \mathcal T  
\ar[d]^-{<(\mathcal A _{\psi _{\Lambda , y_0}})>} \ar[r]^-p&  
f_0(M(X)(-d+n)[-2d+2n]) \cong \varphi _{d-n}(M(X))  \ar[d]^-{\Psi ^n _{\mathrm{ab}}}\\
\mathbf{h}_{-1} ( \mathcal R )[-1]  \ar[r]_-{<e_\psi >}& 
 \mathfrak{Ab} \ar[r]_-{\xi _{\mathrm{ab}}^n}& \mathcal E_{\mathrm{ab}}^n}
 \end{split}
\end{align}

\subsubsection{}    \label{ss.ab.HI}
Notice that $\mathcal E ^n _{\mathrm{ab}} \in \HINST{k}$,
since $\{ \mathbf{h}_i :\DMeff{k}\rightarrow
\HINST{k}, i\in \mathbb Z\}$ is a cohomological functor
 \cite{MR751966}*{Thm. 1.3.6} and $\mathfrak{Ab} , \mathbf h _{-1}(\mathcal R ) \in
 \HINST{k}$.  Furthermore, we observe that $\mathfrak{Ab}$ is a birational sheaf
(being the product of birational sheaves \eqref{def.prod.abvar}), so, by an
argument parallel to \eqref{ss.ext.HINST} 
we deduce that $\mathcal E ^n _{\mathrm{ab}}$ 
is a birational sheaf as well.

We observe that \eqref{diag.main.abeq} induces the
following commutative diagram of abelian groups:
\begin{align}
\begin{split}
\xymatrix{ \Hom_{\DMeff{k}}(\sphere , \mathcal T) \ar[r]^-{p_\ast} 
\ar[d]^-{<(\mathcal A _{\psi _{\Lambda , y_0}})> _\ast}
& \Hom_{\DMeff{k}}(\sphere , \varphi _{d-n}(M(X))) \ar[d]^-{\Psi ^n _{\mathrm{ab} \ast}} 
\overset{(1)}{\cong} CH^n(X)\\
 \Hom_{\DMeff{k}}(\sphere ,
\mathfrak{Ab}) \ar[r]_-{\xi ^n _{\mathrm{ab} \ast}} \ar[d]_-{\cong}^-{(2)}& 
\Hom_{\DMeff{k}}(\sphere , \mathcal E ^n _{\mathrm{ab}}) \cong \Gamma (k , 
\mathcal E ^n _{\mathrm{ab}})\\
\prod \limits_{\substack{\psi \in Reg ^n (X,A) \\
	A \in Abv_k}}
A(k) &}
\end{split}
\end{align}
where the isomorphisms $(1)$, $(2)$ follow, respectively, from
\eqref{rmk.comp.indmap.a} and the construction 
of $\mathfrak{Ab}$ \eqref{def.prod.abvar} as a product of
functor of points
for  abelian varieties $A\in Abv_k$ \eqref{ss.nt.abvar}.
Then, by \eqref{rmk.comp.indmap.b},
\eqref{action.corr} and \eqref{prop.indmap.ev}, 
we obtain the following commutative diagram of abelian groups:
\begin{align}  \label{diag.final.ind.abeq}
\begin{split}
\xymatrix{0 \ar[r]& CH^n_{\mathrm{alg}}(X) \ar[r]^-{p_\ast}
 \ar[d]_-{< \psi >} & CH^n (X) 
\ar[d]^-{\Psi ^n _{\mathrm{ab} \ast}} \\
0 \ar[r]& 
\prod \limits_{\substack{\psi \in Reg ^n (X,A) \\
	A \in Abv_k}}
A(k) \ar[r]_-{\xi ^n _{\mathrm{ab} \ast}} & \Gamma (k, \mathcal E ^n _{\mathrm{ab}})}
\end{split}
\end{align}
with $\pi _{A, \psi}\circ < \psi >=\psi$, where $\pi _{A, \psi}:
\prod \limits_{\substack{\psi \in Reg ^n (X,A) \\
	A \in Abv_k}}
A(k)\rightarrow A(k)$ is the projection to the factor indexed by $(\psi , A)$,
and the
injectivity of $\xi ^n _{\mathrm{ab} \ast}$ follows from the distinguished triangle in the bottom
row of
 \eqref{diag.main.abeq}, since
$0= \Hom _{\DMeff{k}}(\sphere , \mathbf h _{-1}(\mathcal R )[-1])$.

\subsubsection{}  \label{ker.abeq}
In particular, we observe that $\gamma \in \mathrm{ker}(<\psi >)$ if and only
if  $\psi (\gamma)=0 \in A(k)$,
for every abelian variety $A\in SmProj_k$ and every
regular homomorphism $\psi \in Reg^n(X,A)$ 
\eqref{def.gp.reg.hom}.  Hence, $\mathrm{ker}(<\psi >)=CH^n_{
\mathrm{ab}}(X)$
\eqref{def.ab.eq}.

\begin{prop}  \label{prop.comp.ker.ab}
With the notation and conditions of \eqref{def.ab.eq} and
\eqref{ss.1st.step.abeq}-\eqref{diag.final.ind.abeq}.
Then, $\mathrm{ker}(<\psi >)=\mathrm{ker}(\Psi ^n _{\mathrm{ab} \ast})$.
\end{prop}
\begin{proof}
This follows by an argument parallel to \eqref{prop.comp.ker}.
\end{proof}

\subsubsection{}  \label{not.thm1.ab}
With the notation and conditions of \eqref{def.ab.eq} and
\eqref{ss.1st.step.abeq}-\eqref{diag.final.ind.abeq}.  We will write
$M^n_{\mathrm{ab}}(X)\in \DMeff{k}$ for
$\mathcal K _{\mathrm{ab}}^n(X)(d-n)[2d-2n]$ \eqref{diag.main.abeq}, and let
$c_{\mathrm{ab}}^n: M^n_{\mathrm{ab}}(X)\rightarrow f_{d-n}M(X)$ be the map
$k^n_{\mathrm{ab}}(d-n)[2d-2n]$ in $\DMeff{k}$ (see \eqref{diag.main.abeq} and
\cite{MR3614974}*{3.3.3(2)}).

Let $z_{\mathrm{ab}}^n: M^n_{\mathrm{ab}}(X) \rightarrow M(X)$ be the following composition:
\begin{align*}
\xymatrix{M^n_{\mathrm{ab}}(X) \ar[r]^-{c_{\mathrm{ab}}^n}& f_{d-n}M(X) 
\ar[r]^-{\epsilon _{d-n}^{M(X)}}& M(X)}
\end{align*}
where $\epsilon _{d-n}$ is the counit of the adjunction $(i_{d-n}, r_{d-n})$
\eqref{eq.slice.filtration}.
Now, we consider
the  morphism of abelian groups induced by $z _{\mathrm{ab}}^n$:
\begin{align}  \label{diag.main.indmapAB}
\begin{split}
\xymatrix{\Hom _{\DMeff{k}}(\sphere (d-n)[2d-2n],
M^n_{\mathrm{ab}}(X)) \ar[d]^-{z _{\mathrm{ab} \ast}^n} \\
\Hom _{\DMeff{k}}(\sphere (d-n)[2d-2n], M(X))}
\end{split}
\end{align}

\begin{thm}   \label{thm.1.ab}
With the notation and conditions of \eqref{def.ab.eq} and
\eqref{ss.1st.step.abeq}-\eqref{diag.main.indmapAB}.
Under the isomorphism
$\Hom _{\DMeff{k}}(\sphere (d-n)[2d-2n], M(X))\cong CH^n(X)$
\cite{MR1764202}*{Prop. 2.1.4, Thm. 3.2.6},
the image of
$z _{\mathrm{ab} \ast}^n$  \eqref{diag.main.indmapAB} is the subgroup
of algebraic cycles abelian equivalent to zero:
$CH^n _{\mathrm{ab}}(X)\subseteq CH^n_{\mathrm{alg}}(X)$ \eqref{def.ab.eq}.
\end{thm}
\begin{proof}
We recall that  the right column in
\eqref{diag.main.abeq} is a distinguished triangle in $\DMeff{k}$.  Then,
since $z_{\mathrm{ab}}^n = \epsilon _{d-n}^{M(X)} \circ k_{\mathrm{ab}}^n 
(d-n)[2d-2n]$,
we deduce the result by combining
\eqref{ker.abeq}-\eqref{prop.comp.ker.ab}, Voevodsky's cancellation theorem
\cite{MR2804268} and the universal property of the counit 
$\epsilon _{d-n}^{M(X)}:f_{d-n}M(X)\rightarrow M(X)$ \cite{MR3614974}*{3.3.1}. 
\end{proof}

\subsection{Effective covers and slices}
Consider the distinguished triangle in 
the right column of \eqref{diag.main.abeq}.  Since $f_m:\DM{k} \rightarrow
\DM{k}$ is a triangulated functor \eqref{eq.slice.filtration},
we obtain the following distinguished triangle in $\DM{k}$:
\begin{align}  \label{diag.dist.tr.thm2.ab}
\xymatrix{f_m(\mathcal K _{\mathrm{ab}}^n (X))\ar[rr]^-{f_m(k_{\mathrm ab}^n)} &&f _m(\varphi _{d-n}(M(X)))
\ar[rr]^-{f_m(\Psi ^n _{\mathrm{ab}})} && f_m (\mathcal E _{\mathrm{ab}}^n).}
\end{align}

\begin{thm}  \label{thm2.ab}
With the notation and conditions of 
\eqref{def.ab.eq},
\eqref{ss.1st.step.abeq}-\eqref{diag.main.indmapAB} and \eqref{diag.dist.tr.thm2.ab}.
\begin{enumerate}
\item \label{thm2.a.ab}For $m\geq 1$, $f_m(\mathcal E _{\mathrm{ab}}^n)\cong 0$.
\item  \label{thm2.b.ab}For $r\geq 1$, 
the map $f_{d-n+r}(z_{\mathrm{ab}}^n ): 
f_{d-n+r}(M^n_{\mathrm{ab}}(X))  \rightarrow
f_{d-n+r}(M(X))$ is an isomorphism in $\DMeff{k}$
\eqref{not.thm1.ab}.
\item  \label{thm2.c.ab}For $r\geq 1$, 
the map $s_{d-n+r}(z_{\mathrm{ab}}^n ):
s_{d-n+r}(M^n_{\mathrm{ab}}(X))  \rightarrow
s_{d-n+r}(M(X))$ is an isomorphism in $\DMeff{k}$
\eqref{diag.slice.tr}.
\end{enumerate}
\end{thm}
\begin{proof}
\eqref{thm2.a.ab}:  This follows by an argument parallel to
\ref{thm2}\eqref{thm2.a}, observing that  
$\mathcal E _{\mathrm{ab}}^n \in \HINST{k}$ 
is a birational sheaf \eqref{ss.ab.HI}.

\eqref{thm2.b.ab}-\eqref{thm2.c.ab}: This
follows from \ref{thm2.ab}\eqref{thm2.a.ab} combined with
an argument parallel to
\ref{thm2}\eqref{thm2.b} and \ref{thm2}\eqref{thm2.c}, respectively.
\end{proof}


\bibliography{biblio_ab-equiv}

\begin{bibdiv}
\begin{biblist}

\bib{MR2494373}{article}{
      author={Ayoub, Joseph},
      author={Barbieri-Viale, Luca},
       title={1-motivic sheaves and the {A}lbanese functor},
        date={2009},
        ISSN={0022-4049},
     journal={J. Pure Appl. Algebra},
      volume={213},
      number={5},
       pages={809\ndash 839},
         url={http://dx.doi.org/10.1016/j.jpaa.2008.10.002},
      review={\MR{2494373 (2010a:14025)}},
}

\bib{MR2438151}{article}{
      author={Ayoub, Joseph},
       title={Les six op\'erations de {G}rothendieck et le formalisme des
  cycles \'evanescents dans le monde motivique. {II}},
        date={2007},
        ISSN={0303-1179},
     journal={Ast\'erisque},
      number={315},
       pages={vi+364 pp. (2008)},
      review={\MR{2438151}},
}

\bib{MR2735752}{article}{
      author={Ayoub, Joseph},
       title={The {$n$}-motivic {$t$}-structures for {$n=0$}, {$1$} and {$2$}},
        date={2011},
        ISSN={0001-8708},
     journal={Adv. Math.},
      volume={226},
      number={1},
       pages={111\ndash 138},
         url={http://dx.doi.org.sci-hub.io/10.1016/j.aim.2010.06.011},
      review={\MR{2735752}},
}

\bib{MR751966}{incollection}{
      author={Be{\u \i}linson, A.~A.},
      author={Bernstein, J.},
      author={Deligne, P.},
       title={Faisceaux pervers},
        date={1982},
   booktitle={Analysis and topology on singular spaces, {I} ({L}uminy, 1981)},
      series={Ast\'erisque},
      volume={100},
   publisher={Soc. Math. France, Paris},
       pages={5\ndash 171},
      review={\MR{751966}},
}

\bib{MR1780498}{article}{
      author={Beke, Tibor},
       title={Sheafifiable homotopy model categories},
        date={2000},
        ISSN={0305-0041},
     journal={Math. Proc. Cambridge Philos. Soc.},
      volume={129},
      number={3},
       pages={447\ndash 475},
         url={http://dx.doi.org/10.1017/S0305004100004722},
      review={\MR{1780498 (2001i:18015)}},
}

\bib{MR2399083}{article}{
      author={Beilinson, Alexander},
      author={Vologodsky, Vadim},
       title={A {DG} guide to {V}oevodsky's motives},
        date={2008},
        ISSN={1016-443X,1420-8970},
     journal={Geom. Funct. Anal.},
      volume={17},
      number={6},
       pages={1709\ndash 1787},
         url={https://doi.org/10.1007/s00039-007-0644-5},
      review={\MR{2399083}},
}

\bib{MR1644323}{book}{
      author={Fulton, William},
       title={Intersection theory},
     edition={Second},
      series={Ergebnisse der Mathematik und ihrer Grenzgebiete. 3. Folge. A
  Series of Modern Surveys in Mathematics [Results in Mathematics and Related
  Areas. 3rd Series. A Series of Modern Surveys in Mathematics]},
   publisher={Springer-Verlag},
     address={Berlin},
        date={1998},
      volume={2},
        ISBN={3-540-62046-X; 0-387-98549-2},
         url={http://dx.doi.org/10.1007/978-1-4612-1700-8},
      review={\MR{1644323 (99d:14003)}},
}

\bib{MR4486247}{article}{
      author={Galindo, Cesar},
      author={Pelaez, Pablo},
       title={On the convergence of the orthogonal spectral sequence},
        date={2022},
        ISSN={1532-0073,1532-0081},
     journal={Homology Homotopy Appl.},
      volume={24},
      number={2},
       pages={389\ndash 401},
      review={\MR{4486247}},
}

\bib{MR0233825}{article}{
      author={Griffiths, Phillip~A.},
       title={Periods of integrals on algebraic manifolds. {II}. {L}ocal study
  of the period mapping},
        date={1968},
        ISSN={0002-9327,1080-6377},
     journal={Amer. J. Math.},
      volume={90},
       pages={805\ndash 865},
         url={https://doi.org/10.2307/2373485},
      review={\MR{233825}},
}

\bib{MR0309937}{incollection}{
      author={Griffiths, Phillip~A.},
       title={Some transcendental methods in the study of algebraic cycles},
        date={1971},
   booktitle={Several complex variables, {II} ({P}roc. {I}nternat. {C}onf.,
  {U}niv. {M}aryland, {C}ollege {P}ark, {M}d., 1970)},
      series={Lecture Notes in Math.},
      volume={Vol. 185},
   publisher={Springer, Berlin-New York},
       pages={1\ndash 46},
      review={\MR{309937}},
}

\bib{MR463157}{book}{
      author={Hartshorne, Robin},
       title={Algebraic geometry},
      series={Graduate Texts in Mathematics},
   publisher={Springer-Verlag, New York-Heidelberg},
        date={1977},
      volume={No. 52},
        ISBN={0-387-90244-9},
      review={\MR{463157}},
}

\bib{MR1944041}{book}{
      author={Hirschhorn, Philip~S.},
       title={Model categories and their localizations},
      series={Mathematical Surveys and Monographs},
   publisher={American Mathematical Society},
     address={Providence, RI},
        date={2003},
      volume={99},
        ISBN={0-8218-3279-4},
      review={\MR{1944041 (2003j:18018)}},
}

\bib{MR2249535}{article}{
      author={Huber, Annette},
      author={Kahn, Bruno},
       title={The slice filtration and mixed {T}ate motives},
        date={2006},
        ISSN={0010-437X},
     journal={Compos. Math.},
      volume={142},
      number={4},
       pages={907\ndash 936},
         url={http://dx.doi.org.sci-hub.io/10.1112/S0010437X06002107},
      review={\MR{2249535}},
}

\bib{MR1860878}{article}{
      author={Hovey, Mark},
       title={Spectra and symmetric spectra in general model categories},
        date={2001},
        ISSN={0022-4049},
     journal={J. Pure Appl. Algebra},
      volume={165},
      number={1},
       pages={63\ndash 127},
         url={http://dx.doi.org/10.1016/S0022-4049(00)00172-9},
      review={\MR{1860878 (2002j:55006)}},
}

\bib{MR0424807}{incollection}{
      author={Kleiman, Steven~L.},
       title={Finiteness theorems for algebraic cycles},
        date={1971},
   booktitle={Actes du {C}ongr\`es {I}nternational des {M}ath\'{e}maticiens
  ({N}ice, 1970), {T}ome 1},
   publisher={Gauthier-Villars \'{E}diteur, Paris},
       pages={445\ndash 449},
      review={\MR{424807}},
}

\bib{MR3737321}{article}{
      author={Kahn, Bruno},
      author={Sujatha, Ramdorai},
       title={Birational motives, {II}: {T}riangulated birational motives},
        date={2017},
        ISSN={1073-7928,1687-0247},
     journal={Int. Math. Res. Not. IMRN},
      number={22},
       pages={6778\ndash 6831},
         url={https://doi.org/10.1093/imrn/rnw184},
      review={\MR{3737321}},
}

\bib{MR238857}{article}{
      author={Lieberman, David~I.},
       title={Higher {P}icard varieties},
        date={1968},
        ISSN={0002-9327,1080-6377},
     journal={Amer. J. Math.},
      volume={90},
       pages={1165\ndash 1199},
         url={https://doi.org/10.2307/2373295},
      review={\MR{238857}},
}

\bib{MR777590}{article}{
      author={Murre, Jacob~P.},
       title={Un r\'esultat en th\'eorie des cycles alg\'ebriques de
  codimension deux},
        date={1983},
        ISSN={0249-6291},
     journal={C. R. Acad. Sci. Paris S\'er. I Math.},
      volume={296},
      number={23},
       pages={981\ndash 984},
      review={\MR{777590}},
}

\bib{MR805336}{incollection}{
      author={Murre, J.~P.},
       title={Applications of algebraic {$K$}-theory to the theory of algebraic
  cycles},
        date={1985},
   booktitle={Algebraic geometry, {S}itges ({B}arcelona), 1983},
      series={Lecture Notes in Math.},
      volume={1124},
   publisher={Springer, Berlin},
       pages={216\ndash 261},
         url={https://doi.org/10.1007/BFb0075002},
      review={\MR{805336}},
}

\bib{MR1061525}{article}{
      author={Murre, J.~P.},
       title={On the motive of an algebraic surface},
        date={1990},
        ISSN={0075-4102,1435-5345},
     journal={J. Reine Angew. Math.},
      volume={409},
       pages={190\ndash 204},
         url={https://doi.org/10.1515/crll.1990.409.190},
      review={\MR{1061525}},
}

\bib{MR2242284}{book}{
      author={Mazza, Carlo},
      author={Voevodsky, Vladimir},
      author={Weibel, Charles},
       title={Lecture notes on motivic cohomology},
      series={Clay Mathematics Monographs},
   publisher={American Mathematical Society},
     address={Providence, RI},
        date={2006},
      volume={2},
        ISBN={978-0-8218-3847-1; 0-8218-3847-4},
      review={\MR{2242284 (2007e:14035)}},
}

\bib{MR1812507}{book}{
      author={Neeman, Amnon},
       title={Triangulated categories},
      series={Annals of Mathematics Studies},
   publisher={Princeton University Press},
     address={Princeton, NJ},
        date={2001},
      volume={148},
        ISBN={0-691-08685-0; 0-691-08686-9},
      review={\MR{MR1812507 (2001k:18010)}},
}

\bib{MR1308405}{article}{
      author={Neeman, Amnon},
       title={The {G}rothendieck duality theorem via {B}ousfield's techniques
  and {B}rown representability},
        date={1996},
        ISSN={0894-0347},
     journal={J. Amer. Math. Soc.},
      volume={9},
      number={1},
       pages={205\ndash 236},
      review={\MR{MR1308405 (96c:18006)}},
}

\bib{MR3614974}{article}{
      author={Pelaez, Pablo},
       title={Mixed motives and motivic birational covers},
        date={2017},
        ISSN={0022-4049},
     journal={J. Pure Appl. Algebra},
      volume={221},
      number={7},
       pages={1699\ndash 1716},
         url={https://doi.org/10.1016/j.jpaa.2016.12.024},
      review={\MR{3614974}},
}

\bib{pelaez2025incidenceequivalenceblochbeilinsonfiltration}{article}{
      author={Pelaez, Pablo},
      author={Reyes, Araceli},
       title={Incidence equivalence and the {B}loch-{B}eilinson filtration},
        date={2025},
      eprint={https://arxiv.org/abs/2501.19147},
         url={https://arxiv.org/abs/2501.19147},
}

\bib{MR542191}{article}{
      author={Saito, Hiroshi},
       title={Abelian varieties attached to cycles of intermediate dimension},
        date={1979},
        ISSN={0027-7630,2152-6842},
     journal={Nagoya Math. J.},
      volume={75},
       pages={95\ndash 119},
         url={http://projecteuclid.org/euclid.nmj/1118785866},
      review={\MR{542191}},
}

\bib{MR116010}{incollection}{
      author={Samuel, Pierre},
       title={Relations d'\'equivalence en g\'eom\'etrie alg\'ebrique},
        date={1960},
   booktitle={Proc. {I}nternat. {C}ongress {M}ath. 1958},
   publisher={Cambridge Univ. Press, New York},
       pages={470\ndash 487},
      review={\MR{116010}},
}

\bib{MR1957261}{article}{
      author={Spie\ss, Michael},
      author={Szamuely, Tam\'as},
       title={On the {A}lbanese map for smooth quasi-projective varieties},
        date={2003},
        ISSN={0025-5831,1432-1807},
     journal={Math. Ann.},
      volume={325},
      number={1},
       pages={1\ndash 17},
         url={https://doi.org/10.1007/s00208-002-0359-8},
      review={\MR{1957261}},
}

\bib{MR3590347}{article}{
      author={Suslin, Andrei},
       title={Motivic complexes over nonperfect fields},
        date={2017},
        ISSN={2379-1683},
     journal={Ann. K-Theory},
      volume={2},
      number={2},
       pages={277\ndash 302},
         url={https://doi.org/10.2140/akt.2017.2.277},
      review={\MR{3590347}},
}

\bib{MR1744945}{incollection}{
      author={Suslin, Andrei},
      author={Voevodsky, Vladimir},
       title={Bloch-{K}ato conjecture and motivic cohomology with finite
  coefficients},
        date={2000},
   booktitle={The arithmetic and geometry of algebraic cycles ({B}anff, {AB},
  1998)},
      series={NATO Sci. Ser. C Math. Phys. Sci.},
      volume={548},
   publisher={Kluwer Acad. Publ., Dordrecht},
       pages={117\ndash 189},
      review={\MR{1744945 (2001g:14031)}},
}

\bib{MR1764200}{incollection}{
      author={Voevodsky, Vladimir},
       title={Cohomological theory of presheaves with transfers},
        date={2000},
   booktitle={Cycles, transfers, and motivic homology theories},
      series={Ann. of Math. Stud.},
      volume={143},
   publisher={Princeton Univ. Press},
     address={Princeton, NJ},
       pages={87\ndash 137},
      review={\MR{1764200}},
}

\bib{MR1764202}{incollection}{
      author={Voevodsky, Vladimir},
       title={Triangulated categories of motives over a field},
        date={2000},
   booktitle={Cycles, transfers, and motivic homology theories},
      series={Ann. of Math. Stud.},
      volume={143},
   publisher={Princeton Univ. Press},
     address={Princeton, NJ},
       pages={188\ndash 238},
      review={\MR{1764202}},
}

\bib{MR1883180}{article}{
      author={Voevodsky, Vladimir},
       title={Motivic cohomology groups are isomorphic to higher {C}how groups
  in any characteristic},
        date={2002},
        ISSN={1073-7928},
     journal={Int. Math. Res. Not.},
      number={7},
       pages={351\ndash 355},
         url={http://dx.doi.org/10.1155/S107379280210403X},
      review={\MR{1883180 (2003c:14021)}},
}

\bib{MR1977582}{incollection}{
      author={Voevodsky, Vladimir},
       title={Open problems in the motivic stable homotopy theory. {I}},
        date={2002},
   booktitle={Motives, polylogarithms and hodge theory, part i (irvine, ca,
  1998)},
      series={Int. Press Lect. Ser.},
      volume={3},
   publisher={Int. Press, Somerville, MA},
       pages={3\ndash 34},
      review={\MR{MR1977582 (2005e:14030)}},
}

\bib{MR2804268}{article}{
      author={Voevodsky, Vladimir},
       title={Cancellation theorem},
        date={2010},
        ISSN={1431-0635},
     journal={Doc. Math.},
      number={Extra volume: Andrei A. Suslin sixtieth birthday},
       pages={671\ndash 685},
      review={\MR{2804268 (2012d:14035)}},
}

\bib{MR2600283}{article}{
      author={Voevodsky, Vladimir},
       title={Motives over simplicial schemes},
        date={2010},
        ISSN={1865-2433},
     journal={J. K-Theory},
      volume={5},
      number={1},
       pages={1\ndash 38},
         url={http://dx.doi.org/10.1017/is010001030jkt107},
      review={\MR{2600283}},
}

\bib{MR0050330}{article}{
      author={Weil, Andr\'{e}},
       title={On {P}icard varieties},
        date={1952},
        ISSN={0002-9327,1080-6377},
     journal={Amer. J. Math.},
      volume={74},
       pages={865\ndash 894},
         url={https://doi.org/10.2307/2372230},
      review={\MR{50330}},
}

\end{biblist}
\end{bibdiv}
\bibliographystyle{abbrv}

\end{document}